\documentclass[11pt]{amsart}
\usepackage{amssymb, amsmath, amsthm,amssymb,amsfonts,latexsym}
\usepackage{tikz-cd}
\usepackage{extarrows}
\usepackage{mathrsfs}
\usepackage{arydshln}
\usepackage[figuresleft]{rotating}
\usepackage{multirow}
\usepackage{hyperref}
\usepackage{indentfirst}
\usepackage{enumerate}
\usepackage{booktabs}

\usepackage{pdfsync}
\usepackage[square, comma, sort&compress, numbers]{natbib}

\theoremstyle{plain}
\newtheorem{theorem}{Theorem}[section]
\newtheorem{proposition}[theorem]{Proposition}
\newtheorem{corollary}[theorem]{Corollary}
\newtheorem{lemma}[theorem]{Lemma}

\theoremstyle{definition}

\newtheorem{example}[theorem]{Example}

\theoremstyle{remark}
\newtheorem{remark}[theorem]{Remark}

\newcommand{\Z}{\ensuremath{\mathbb{Z}}}
\newcommand{\R}{\ensuremath{\mathbb{R}}}

\newcommand{\HP}{\ensuremath{\mathbb{H}P}}
\newcommand{\FP}{\ensuremath{\mathbb{F}P}}
\newcommand{\z}[1]{\ensuremath{\Z/2^{#1}}}
\newcommand{\zp}[1]{\ensuremath{\Z/p^{#1}}}
\newcommand{\zz}[1]{\ensuremath{\mathbb{Z}_{(#1)}}}

\newcommand{\EM}{Eilenberg-MacLane\,}
\newcommand{\s}[1]{\ensuremath{S^{#1}\langle #1\rangle }}

\newcommand{\xra}{\ensuremath{\xrightarrow}}
\newcommand{\coker}{\ensuremath{\mathrm{coker\,}}}
\newcommand{\sq}[1]{\ensuremath{\mathrm{Sq}_{#1}^2}}

\newcommand{\pzpr}{\ensuremath{\mathcal{P}^1_{\zp{r}}}}
\newcommand{\T}[1]{\ensuremath{T^{#1}_{\infty}(p^r)}}

\newcommand{\B}{\ensuremath{\mathcal{B}}}

\begin{document}
\title[Modular Cohomotopy]{On Modular Cohomotopy Groups}
\author[P. Li]{Pengcheng Li}
\address{Department of Mathematics, Southen University of Science and Technology, Shenzhen {\rm 518055, P. R. China}}
\email{lipc@sustech.edu.cn}

\author[J. Pan]{Jianzhong Pan}
\address{
Institute of Mathematics,  Chinese Academy of Mathematics and Systems Science,  University of Chinese Academy of Sciences, Beijing {\rm 100190, P. R. China}}
\email{pjz@amss.ac.cn}

\author[J. Wu]{Jie Wu}
\address{
   School of Mathematical Sciences, Hebei Normal University,  Shijiazhuang,  {\rm 050024, P. R. China}; Yanqi Lake Beijing Institute of Mathematica Sciences, Yanqihu, Huairou District, Beijing {\rm 101408, P. R. China}
   }
\email{wujie@bimsa.cn}

\begin{abstract}
   Let $p$ be a prime and let $\pi^n(X;\mathbb{Z}/p^r)=[X,M_n(\mathbb{Z}/p^r)]$ be the set of homotopy classes of based maps from CW-complexes $X$ into the mod $p^r$  Moore spaces $M_n(\mathbb{Z}/p^r)$ of degree $n$, where $\mathbb{Z}/p^r$ denotes the integers mod $p^r$. In this paper we firstly determine the modular cohomotopy groups $\pi^n(X;\mathbb{Z}/p^r)$ up to extensions by classical methods of primary cohomology operations and give conditions for the splitness of the extensions. Secondly we utilize some unstable homotopy theory of Moore spaces to study the modular cohomotopy groups; especially, the group $\pi^3(X;\mathbb{Z}_{(2)})$ with $\dim(X)\leq 6$ is determined.
 \end{abstract}
 
 \maketitle

 \section{Introduction}
 
 The cohomotopy groups $\pi^n(X)=[X,S^n]$ of homotopy classes of maps into $n$-sphere with $\dim(X)\leq 2n-2$ were firstly introduced by Borsuk \cite{Borsuk36} and systematically studied by Spanier \cite{Spanier49}. Hereafter, Peterson \cite{Peterson56-2} defined the \emph{generalized (stable) cohomotopy groups} of a CW-pair $(X,A)$ of dimension at most $2n-2$ by
  \[\pi^n(X,A;G):=[X,A;M_n(G),\ast],\quad\pi^n(X;G):=[X,M_n(G)],\]where $M_n(G)$ denotes the Moore space with the (reduced) homology group $G$ precisely only in dimension $n$.  Clearly $\pi^n(X,A;G)\cong \pi^{n}(X/A;G)$ and $\pi^n(X;\Z)\cong \pi^n(X)$. Peterson's main related results are summarized in Theorem \ref{thmPeterson}.
 
 Due to the \emph{Pontrjagin-Thom construction} (cf.~\citep[Chapter IX, Theorem 5.5]{Kosinski93}), the cohomotopy theory of closed smooth manifolds has been extensively studied and is still a hot research direction, for instance \cite{Hansen81,Taylor2012,KMT2012,Konstantis2020,Kons2020}.
 By Bauer and Furuta's famous work \cite{BF04}, the stable cohomotopy groups of the complex projective space $\mathbb{C}P^{d-1}(d>1)$ have close relationship with the cohomotopy Seiberg-Witten invariants for a Riemannian four-manifold with the first Betti number $b_1=0$.
 Cohomotopy theory also boosts the development of differential geometry and mathematical physics. In 2005, Auckly and Kapitanski \cite{AK05} gave an analytic proof of Pontrjagin's classification of maps from a three-manifold $M$ to $S^2$ \cite{Pontrjagin1941} by the method of smooth flat connections on $M$. Restricted on those flat connections with finite energy, they further proved the existence of a minimizer of the energy functional $E_\varphi$ introduced by Faddeev \cite{KF75} of a reference map $\varphi\colon M\to S^2$.

 Let $\zp{r}$ denote the group of mod $p^r$ integers for some prime $p$ and some positive integer $r$. In this paper we study the cohomotopy groups from the new perspective of the \emph{modular cohomotopy groups $\pi^\ast(X;\zp{r})$}.
 In general, modular cohomotopy groups detect the $p$-torsion information of the integral cohomotopy groups, hence to a certain extend, modular cohomotopy groups are useful and advantageous to the computations of the (torsion of) integral ones, which as stated before have close connection with geometry and physics. On the other hand, the following example shows a direct bond between modular cohomotopy and geometric topology.  The equivalence classes of oriented real $n$-dimensional vector bundles over a CW-complex $X$ are classified by $[X,\B SO(n)]$, where $\B SO(n)$ is the classifying space of  the special orthogonal group $SO(n)$. Since the Stiefel manifold $V_{2n+1,2}$ of orthogonal $2$-frames in $\R^{2n+1}$ is a $(4n-1)$-dimensional CW-complex with the $(4n-2)$-skeleton $M_{2n-1}(\z{})$ (cf.~\citep[page 302]{Hatcher}), there is an exact sequence
 \[\cdots\to [X,SO(2n+1)]\to \pi^{2n-1}(X;\z{})\xra{}[X,\B  SO(2n-1)]\]
  if $\dim(X)\leq 4n-3$.  The group $\pi^{2n-1}(X;\z{})$ is closely linked with the classification of vector bundles and tells information that cannot be obtained by the classic tools of characteristic classes and $K$-theory, which are appropriate only when $n$ passes to $\infty$.

 To make more sense of the introduction and formulate the main theorems, we need some global conventions and notations.
 Unless otherwise stated, complexes mean connected based CW-complexes and maps are base-point-preserving; we don't notationally distinguish a map with its homotopy class.
 For an abelian group $G$ and an integer $n$, $H^n(X;G)$ denotes the $n$-{th} \emph{reduced} cohomology group with coefficients in $G$, $H^n(X)=H^n(X;\Z)$; similar notations are used for homology.  For a prime $p$, let
 \[\mathcal{P}^1\colon H^{\ast}(-;\zp{})\to H^{\ast+2p-2}(-;\zp{})\]
 be the Steenrod's reduced power operations and by convention, $\mathcal{P}^1$ refers to the Steenrod square $\mathrm{Sq}^{2}$ if $p=2$.
 For $G=\Z,\zp{r}$ or $\zz{p}$, define $\mathcal{P}^1_G$ by the composition
 \begin{equation}\label{StG}
    \mathcal{P}^1_{G}\colon K_n(G)\xra{}K_n(\zp{})\xra{\mathcal{P}^1}K_{n+2p-2}(\zp{}),
 \end{equation}
 where the first map corresponds to the reduction mod $p$.
 If $f\colon X\to Y$ is a map inducing a homomorphism
 \[f_\sharp \colon [W,X]\to [W,Y]\]
  for a space $W$,  the kernel and cokernel of $f_\sharp$ are respectively denoted by
  \begin{align*}
     K_f(W):&=\ker(f_\sharp)=\{\alpha\colon W\to X:f\circ \alpha= \ast\},\\
     T_{f}(W):&=\coker(f_\sharp)=[W,Y]/f_\sharp[W, X].
  \end{align*}
 
 For a complex $X$ of dimension $\leq 2n-2$, the \emph{cohomotopy Hurewicz homomorphism $h^n=h^n_G$ with coefficients in $G$}, or simply the \emph{generalized cohomotopy Hurewicz homomorphism} is defined by the composition
 \[
       h^n_G\colon  \pi^n(X;G)=[X,M_n(G)]\xra{(i_M)_\sharp}[X,K_n(G)]\cong H^{n}(X;G),
 \]
 where $i_M\colon M_n(G)\to K_n(G)$ is the canonical inclusion map into the \EM space $K_n(G)$ of type $(G,n)$ and the isomorphism is due to the Brown's representability theorem (cf.~\citep[Theorem 4.57]{Hatcher}).
 Note that $h^n_G$ is natural with respect to maps $X\to Y$ and group homomorphisms $G\to H$.
 Let $p$ be a prime and let $\zz{p}$ be the group of $p$-local integers.
 If $G=\Z,\zp{r}$ or $\zz{p}$, then $H^n(M_n(G);G)\cong G\langle [i_M]\rangle$ and we have
 \[h^n_G(f)=f^\ast([i_M]),\]
 where the cohomology class $[i_M]$ is represented by the inclusion $i_M$. We call $h^n_\Z,h^n_{\zp{r}},h^n_{\zz{p}}$ the $n$-{th} \emph{integral, mod $p^r$, $p$-local cohomotopy Hurewicz homomorphism}, respectively.

 In \citep[Section 6.1]{Taylor2012}, Taylor gave a modern approach to Steenrod's classification theorem \cite{Steenrod1947}: For a complex $X$ of dimension at most $n+1$, $n\geq 3$,  there is a short exact sequence of abelian groups
 \[0\to T_{\Omega \mathrm{Sq}^2_\Z}(X)\to \pi^{n}(X)\to H^n(X)\to 0,\] where $\mathrm{Sq}^2_\Z\colon H^n(X;\Z)\xra{}H^{n+2}(X;\z{})$. Moreover, using the methods in \cite{LT72}, he showed that the above extension splits if and only if
 \begin{equation*}
    \mathrm{Sq}^2_\Z\big(H^{n-1}(X;\Z)\big)
    =\mathrm{Sq}^2\big(H^{n-1}(X;\z{})\big).
 \end{equation*}
 
 Our first main theorem is partially motivated by Taylor's analysis.
 
 \begin{theorem}\label{mainthm1}
 Let $p$ be a prime and $X$ be a complex of dimension $\leq n+2p-3,n\geq 2p-1$. For $G=\zz{p}$ or $\zp{r}(r\geq 1)$, there is a short exact sequence of abelian groups:
 \begin{equation}\label{EXT:G}
    0\to T_{\Omega \mathcal{P}^1_G}(X)\to \pi^n(X;G)\xra{h^n} H^n(X;G)\to 0,
 \end{equation}
 where $T_{\Omega \mathcal{P}^1_G}(X)=H^{n+2p-3}(X;\zp{})/\mathcal{P}^1_G\big(H^{n-1}(X;G)\big);$ $\mathcal{P}^1_G=\mathrm{Sq}^2_G$ if $p=2$.
 
 \begin{enumerate}[1.]
    \itemsep1em
    \item If $G=\zz{p}$, $p\geq 2$, then the extension (\ref{EXT:G}) splits if and only if
    \begin{equation}
       \mathcal{P}^1_{\zz{p}}\big(H^{n-1}(X;\zz{p})\big)=\mathcal{P}^1\big(H^{n-1}(X;\zp{})\big).
    \end{equation}

    \item If $G=\z{r}$, then the extension (\ref{EXT:G}) splits if one of the following conditions holds:
    \begin{enumerate}
       \item $r\geq 2$, the homomorphism $[2]\colon H^{n}(X;\z{r-1})\to H^n(X;\z{r})$ induced by the coefficient homomorphism $\z{r-1}\xra{\cdot 2}\z{r}$ is injective and
       \begin{equation}\label{mainthm1:2.a}
          \mathrm{Sq}^1\big(H^{n}(X;\z{})\big)\subseteq\sq{\z{r}}\big(H^{n-1}(X;\z{r})\big).
       \end{equation}
 
      \item  $r=1$,
    $\mathrm{Sq}^1\colon H^n(X;\z{})\to H^{n+1}(X;\z{})$ is trivial.
 \end{enumerate}
   \item If $G=\zp{r}$, $p\geq 3$, then the extension (\ref{EXT:G}) splits if $r=1$, or $r\geq 2$ and the homomorphism $[p]\colon H^n(X;\Z/p^{r-1})\to H^n(X;\zp{r})$ induced by $\zp{r-1}\xra{\cdot p}\zp{r}$ is injective.
 
 \end{enumerate}
 
 \end{theorem}
 
 \begin{remark}\label{rmk:mainthm1}
   \begin{flushleft}
 
   \end{flushleft}
    \begin{enumerate}
       \item (2.a) can be strengthened as follows: If $r\geq 2$ and the homomorphism $[2]$ is injective, then the extension (\ref{EXT:G}) splits if and only if the inclusion (\ref{mainthm1:2.a}) holds.
 
       \item If $G=\zp{}$, $p\geq 3$, one can also get the splitness of (\ref{EXT:G}) by applying Theorem \ref{thmPeterson} (\ref{lemsplit:zp}).
    \end{enumerate}
 
 \end{remark}

 When $n$ is odd, observe that the $n$-connected cover $\s{n}$ of $S^n$ is a torsion space of finite type.  Hence, for a finite suspension $X$, the group $[X, \s{2n+1}]$ is determined by its $p$-torsion components. Let $W_{(p)}$ denote the $p$-localization of a simply-connected space $W$. As can be seen from Section~\ref{sec:mod-p}, there is an isomorphism
 \[[X, \s{2n+1}_{(p)}]\cong \pi^{2n+2p-2}(X;\mathbb{Z}/p)\]
 if $\dim(X)\leq 2n+4p-4$, $p\geq 5,n\geq 1$ or $n=1, p\geq 2$.
 Note that $p$-localization loosens the restriction on $\dim(X)$, especially for greater primes $p$.
 Another way is to consider the double suspension $E^2\colon S^{2n-1}\to \Omega^2S^{2n+1}$, $n\geq 1$.
 Let $C(n)$ be the homotopy fibre of $E^2$; as an example, $C(1)=\Omega^3\s{3}$ \cite{CMN79-2}. Selick \cite{Selick85} showed that $C(n)_{(p)}$ is an H-space for each $n\geq 1, p\geq 2$ and Gray \cite{Gray88} constructed the classfying space $\B C(n)_{(p)}$ for such $p,n$.  If $p$ is odd, $C(n)_{(p)}$ has bottom two cells $M_{2pn-3}(\Z/p)$ (see Lemma \ref{lem:Cn}) and the bond with $\pi^{2pn-3}(X;\zp{})$ is set up.

 \begin{theorem}\label{mainthm:oc}
 Let $n$ be a positive integer and let $p$ be a prime satisfying:  i) $n\geq 1,p\geq 5$;  or ii) $n=1,p\geq 2$.
   \begin{enumerate}[1.]
    \item\label{mainthm:oc-1}  For a complex $X$ of dimension $\leq 2n+4p-4$, there is an exact sequence of groups
  \begin{multline*}\label{ES:oc-1}
       \pi^{2n+1}(\Sigma^2 X;\zz{p})\xra{(\Omega^2\iota_{(p)})_\sharp}  H^{2n+1}(\Sigma^2 X;\zz{p})\to [X,\Omega \s{2n+1}_{(p)}]  \\
       \to\pi^{2n+1}(\Sigma X;\zz{p})\xra{(\Omega\iota_{(p)})_\sharp} H^{2n+1}(\Sigma X;\zz{p})\to\pi^{2n+2p-2}(X;\zp{})\\
       \xra{}\pi^{2n+1}(X;\zz{p})
    \xra{(\iota_{(p)})_\sharp}H^{2n+1}(X;\zz{p}),
 \end{multline*}
 where $\iota_{(p)}\colon S^{2n+1}_{(p)}\to K_{2n+1}(\zz{p})$ is the canonical map representing a generator of $H^{2n+1}(S^{2n+1};\zz{p})$. In particular, if $p\geq 3$,
 \begin{enumerate}
    \item there is a natural isomorphism:
 \[\pi^{2n+2p-2}(X;\zp{})\cong T_{\Omega \iota_{(p)}}(X)\oplus K_{\iota_{(p)}}(X).\]
  \item there is a natural isomorphism if $\dim(X)\leq 2n+4p-5$:
   \[\pi^{2n+2p-3}(X;\zp{})\cong [X,\Omega \s{2n+1}_{(p)}]\cong T_{\Omega^2\iota_{(p)}}(X)\oplus K_{\Omega\iota_{(p)}}(X).\]
 
 \end{enumerate}
 
 \item\label{mainthm:oc-2}For a complex $X$ of dimension $\leq 4p-3$, $p\geq 2$, the exact sequence in (\ref{mainthm:oc-1}) extends from the right by \[\pi^{3}(X;\zz{p})\xra{(\iota_{(p)})_\sharp} H^{3}(X;\zz{p})\xra{\mathcal{P}^1_{\zz{p}}}H^{2p+1}(X;\zp{}),\]
 where $\mathcal{P}^1_{\zz{2}}=\mathrm{Sq}^2_{\zz{2}}$.
 \end{enumerate}
 
 \end{theorem}
 Note that Theorem \ref{mainthm:oc} gives an connection between $\pi^{2n+2p-2}(X;\zp{})$ and $\pi^{2n+1}(X;\zz{p})$.
 If $\pi^{2n+2p-2}(X;\zp{})$ is known, we can apply Theorem \ref{mainthm:oc} to determine $\pi^{2n+1}(X;\zz{p})$ (up to extension).
 
 \begin{corollary}\label{maincor}
 Let $X$ be a complex of dimension $\leq 6$.
 \begin{enumerate}[1.]
 \item There is a short exact sequence of groups:
 \begin{equation*}
    0\to T_{\partial}(X)\to \pi^3(X;\zz{2}) \xra{h^3} H^3(X;\zz{2}) \dashrightarrow \pi^5(X;\z{}),
 \end{equation*}
 where the dashed arrow here means the homomorphism is merely of sets and $T_\partial(X)$ is the cokernel of the homomorphism
 \[\partial_\sharp\colon H^2(X;\zz{2})\to[X,\s{3}_{(2)}]\cong \pi^4(X;\z{})\]
 induced by the connecting map $\partial\colon K_2(\zz{2})\to \s{3}_{(2)}$.

 \item If $\dim(X)\leq 5$, then
    \[
      h^3(\pi^3(X;\zz{2}))=\ker\big(H^3(X;\zz{2})\xra{\mathrm{Sq}^2_{\zz{2}}}H^5(X;\z{})\big)
    \]
    and there is a central group extension
 \begin{equation}\label{EXT:5-cpx}
    0\to T_\partial(X)\xra{j} \pi^3(X;\zz{2}) \xra{h^3} \ker(\mathrm{Sq}^2_{\zz{2}})\to 0,
 \end{equation}
 which is determined by the following two functions $\Gamma,\Phi_2$:
 \begin{enumerate}[(i)]
    \item $\Gamma\colon \ker(\mathrm{Sq}^2_{\zz{2}})\times \ker(\mathrm{Sq}^2_{\zz{2}})\to T_{\partial}(X)$: if $x,y\in \ker(\mathrm{Sq}^2_{\zz{2}})$, let $g,h\in \pi^3(X;\zz{2})$ such that $h^3(g)=x,h^3(h)=y$, then $\Gamma (x,y)=j^{-1}(ghg^{-1}h^{-1})$.
    \item  $\Phi_2\colon \ker(\mathrm{Sq}^2_{\zz{2}})\to T_{\partial}(X)/2T_{\partial}(X)$: If $\bar{\alpha}\in \pi^3(X;\zz{2})$ satisfies $h^3(\bar{\alpha})=\alpha=\delta(\alpha')$ for some $\alpha'\in H^2(X;\z{})$, where $\delta$ is the obvious Bockstein, then
    \begin{equation}\label{Eq:5-cpx}
      \Phi_2(\alpha)=\bar{\alpha}^2=[\partial_2\alpha']\mod 2T_{\partial}(X),
    \end{equation}
    where $\partial_2$ is given by the decomposition
    \begin{equation*}
       \partial_\sharp\colon H^2(X;\zz{2})\xra{\mod 2} H^2(X;\z{})\xra{\partial_{2\sharp}} \pi^4(X;\z{}).
    \end{equation*}
 \end{enumerate}
  \end{enumerate}
 \end{corollary}

 \begin{theorem}\label{mainthm:Cn}
    Let $p\geq 5$ be a prime and let $n\geq 1$. For any complexes $X$ of dimension $\leq 2pn+2n-5$, there is an exact sequence of abelian groups
    \begin{multline*}
       \pi^{2n-1}(\Sigma X;\zz{p})\xra{(\Omega E^2_{(p)})_\sharp} \pi^{2n+1}(\Sigma^3X;\zz{p})
       \xra{} [X,C(n)_{(p)}]\\
       \xra{}\pi^{2n-1}(X;\zz{p})\xra{(E^2_{(p)})_\sharp}\pi^{2n+1}(\Sigma^2X;\zz{p})
       \xra{}\pi^{2pn-2}(X;\zp{}).
    \end{multline*}
 
    In particular, if $\dim(X)\leq 2pn+2n-6$, there is a natural isomorphism:
    \[ \pi^{2pn-3}(X;\zp{})\cong [X,C(n)_{(p)}]\cong T_{\Omega E^2_{(p)}}(X)\oplus K_{E^2_{(p)}}(X).\]
 
 \end{theorem}
 
 Our final result generalizes Peterson's exact sequence theorem (see Theorem \ref{thmPeterson} (\ref{thmUCT})) for cohomotopy with $\zp{r}$-coefficients.
 \begin{theorem}\label{mainthm:GUCT}
    Let $X$ be a finite complex of dimension $\leq 4n-3$, $n\geq 2$. Then for each prime $p\geq 5$,  there is an exact sequence of abelian groups
    \begin{equation*}
       0\to \pi^{2n-1}(X;\zz{p})\otimes\zp{r}\xra{}\pi^{2n-1}(X;\zp{r})\xra{}\pi^{2n+1}(\Sigma X;\zz{p}).
    \end{equation*}
 
    If $\dim(X)\leq 4n-4$, $\pi^{2n-1}(X;\zz{p})\otimes\zp{r}$ is a direct summand of $\pi^{2n-1}(X;\zp{r})$ and isomorphic to $\pi^{2n-1}(X)\otimes\zp{r}$.
 \end{theorem}
 Note that if $\dim(X)=4n-3$, the group structure of $\pi^{2n-1}(X;\zp{r})$ is not obtained by Freudenthal's suspension isomorphism theorem; the details refer to Section~\ref{sec:Anick}.
 
 The paper is organized as follows. Section \ref{sec:prelim} covers Peterson's main results on the generalized cohomotopy groups and some auxiliary lemmas.
 In Section \ref{sec: oldtools} we give the proof of Theorem \ref{mainthm1} and compute partial cohomotopy groups of the projective spaces $\FP^n$ ($\mathbb{F}=\mathbb{C},\mathbb{H}$) and $M_n(\z{s})$ as examples.
 In Section \ref{sec:unstHtp} we combine the classic unstable homotopy theory of Moore spaces to prove the remaining results.

 \section*{Acknowledgement}
 The first author was partially supported by Natural Science Foundation of China under Grant 12101290 and Project funded by China Postdoctoral Science Foundation under Grant 2021M691441, the second author was partially supported by Natural Science Foundation of China under Grant 11971461, and the third author was partially supported by Natural Science Foundation of China under Grant 11971144 and High-level Scientific Research Foundation of Hebei Province.
 The authors would like to thank Daciberg Goncalves for fruitful communications on the topic of cohomotopy, and thank the referee(s) gratefully for useful and detailed comments on a previous version of the manuscript.
 
 \section{Preliminaries}\label{sec:prelim}
 In this section we firstly summarize Peterson's primary work on the generalized cohomotopy groups and then list some lemmas that will be used in the latter sections.
 
 \subsection{Peterson's generalized cohomotopy groups}\label{sec:Peterson}
 
  \begin{theorem}[Peterson, \cite{Peterson56-2}]\label{thmPeterson}
 Let $(X,A)$ be a CW pair of dimension $N\leq 2n-2$, $n\geq 2$. Let $G$ be an abelian group.
 \begin{enumerate}[1.]
    \item\label{thmACH} The generalized cohomotopy groups $\pi^n(X,A;G)$ satisfies all the Eilenberg–Steenrod axioms for cohomology theory \cite{ES52}.
    \item\label{lemsplit:zp} If $G$ is a field without $2$-torsion, then $\pi^n(X,A;G)$ is a vector space over $G$.
    \item\label{thmUCT} If $G$ is finitely generated or $(X,A)$ is finite, then there is a \emph{universal coefficients exact sequence} for each $r> (N+1)/2$:
    \begin{equation}\label{SES}
       0\to \pi^r(X,A)\otimes G\xra{}\pi^r(X,A;G)\xra{}\mathrm{Tor}(\pi^{r+1}(X,A),G)\to 0.
    \end{equation}
    This sequence is natural with respect to maps $f\colon (X,A)\to (Y,B)$ and if $G$ has no $2$-torsion, then it is natural with respect to homomorphisms $\varphi\colon G\to H$. Furthermore, the extension (\ref{SES}) splits if $\pi^r(X,A)$ is finitely generated and $G$ has no $2$-torsion.
 
 \end{enumerate}
 
  \end{theorem}

 \subsection{Some lemmas}\label{sec:lems}  Let $n\geq 0$. Recall that a space $X$ is \emph{$n$-connected} if $\pi_i(X)=0$ for all $i\leq n$; a map $f\colon X\to Y$ is \emph{$n$-connected} if the homotopy fibre of $f$ is $(n-1)$-connected.
 
 The following lemma is well-known as the universal property of the James space $J(X)\simeq \Omega\Sigma X$ \cite{James55}:
 \begin{lemma}\label{lemJamesEXT}
 Let $f\colon X\to Y$ be a map of spaces where $Y$ is a homotopy associative H-space. Then $f$ extends to an H-map $\bar{f}\colon\Omega\Sigma X\to Y$ which is unique up to homotopy.
 \end{lemma}

 \begin{lemma}\label{lemJamesUP}
 Let $f\colon X\to Y$ be an $n$-connected map, where  $X$ is $m$-connected and $Y$ is a homotopy associative H-space, $m,n\geq 1$.  Then for any space $W$ of dimension $\leq \min\{n-1,2m\}$, the induced map
    \[f_\sharp\colon [W,X]\to [W,Y]\]
 is an isomorphism, where the group structure of $[W,X]$ is induced by the loop suspension $E\colon X\to \Omega \Sigma X$ and that of $[W,Y]$ is induced by the H-space structure of $Y$.
 \begin{proof}
    By \citep[Proposition 2.4.6]{Arkowitzbook}, the induced map $f_\sharp\colon [W,X]\to [W,Y]$ is a bijection.
    Since $X$ is $m$-connected, the loop suspension $E=E_X\colon X\to \Omega \Sigma X$ is $(2m+1)$-connected and hence induces a bijection
    \[E_\sharp\colon [W,X]\xra{1:1}[W,\Omega \Sigma X],\]
    which provides $[W,X]$ a group structure.  By Lemma \ref{lemJamesEXT}, the map $f$ extends to an H-map $\bar{f}\colon\Omega\Sigma X\to Y$ and hence the group homomorphism
    \[\bar{f}_\sharp\colon [W,\Omega\Sigma X]\to [W,Y].\]
    Since $\bar{f}_\sharp$ is bijective, it is an isomorphism of groups.
 \end{proof}
 \end{lemma}
 
 
 

 
 \begin{lemma}[Serre, cf.~Theorem 6.4.4 of \cite{Arkowitzbook}]\label{lemserre}
    Let $F\xra{i} E\xra{\pi} B$ be a (homotopy) fibration, in which $F$ is $r$-connected, $B$ is $s$-connected, $r\geq 1,s\geq 0$. Then for any coefficient group $G$, there is an exact sequence of cohomology groups:
    \begin{multline*}
       H^0(B;G)\xra{\pi^\ast}H^0(E;G)\xra{i^\ast}H^0(F;G)\xra{\delta^0}H^1(B;G)\xra{}\cdots\\
       \to H^{N-1}(F)\xra{\delta^{N-1}}H^{N}(B;G)\xra{\pi^\ast}H^N(E;G)\xra{i^\ast}H^N(F;G),
    \end{multline*}
    where $N=r+s+1$.
 \end{lemma}
 
 \begin{lemma}[cf.~Theorem 10.3 of \cite{May70}]\label{lemchEM}
  Let $p$ be a prime, $1\leq r\leq \infty$ and by convention, let $\zp{\infty}=\Z$.   Let $\iota_n\in H^n(K_n(\zp{r});\zp{})$ be the fundamental class.
 $H^\ast(K_n(\zp{r});\zp{})$ is a free, graded-commutative algebra, whose generators $\mathcal{P}_r^I\iota_n$ with associated excess $e=e(I)\leq 5$ are listed as follows:
    \begin{center}
       \begin{tabular}{c|l}
          \hline
          $e$&$\mathcal{P}_r^I$\\
          \hline
          $0$&$id$\\
          \hline
          $1$& $\beta_r,\mathcal{P}^1\beta_r,\cdots$\\
          \hline
          $2$& $\mathcal{P}^1,\beta_1\mathcal{P}^1\beta_r\cdots$\\
          \hline
          $3$& $\beta_1\mathcal{P}^1, \mathcal{P}^2\beta_r,\cdots$\\
          \hline
          $4$& $\mathcal{P}^2,\cdots$\\
          \hline
          $5$& $\beta_1\mathcal{P}^2,\cdots$\\
          \hline
       \end{tabular}
    \end{center}
    where $\beta_r$ is the $r$-{th} Bockstein associated to the short exact sequence
    \[0\to \zp{}\xra{\cdot p^r}\zp{r+1}\to \zp{r}\to 0\]
 for $r<\infty$ and $\beta_\infty=0$. Note if $p=2$, $\beta_1\mathrm{Sq}^i=\mathrm{Sq}^1\mathrm{Sq}^i=(i-1) \mathrm{Sq}^{i+1}$.

 \end{lemma}
 
 \begin{lemma}\label{lemToda}
    Let $p$ be an odd prime, $n\geq 1$. Let \[S^{2n+2p-2}\xra{i_p}M_{2n+2p-2}(\zp{})\xra{j_p}S^{2n+2p-1}\] be the cofibration with $i_p=i_p(2n+2p-2),j_p=j_p(2n+2p-2)$ the canonical inclusion and pinch maps, respectively.
    The nontrivial $p$-primary components of $\pi_{2n+i}(X)$ for $X=S^{2n+1},M_{2n+2p-2}(\Z/p)$ in the range $i<4p-3$ are listed as follows.
 \begin{center}
    \begin{tabular}{c|c|c|c}
       \hline
       $i$&$2p-2$&$4p-5$&$4p-4$\\
       \hline
      $S^{2n+1}$&$\zp{}$&$\zp{}$&$\zp{}$\\
       \hline
       \emph{generators}&$\alpha_1(2n+1)$&$\alpha_1(2n+1)\circ \alpha_1(2n+2p-2)$&$\alpha_2(2n+1)$\\
       \hline\hline
   $M_{2n+2p-2}(\Z/p)$&$\zp{}$&$\zp{}$&$\zp{}$\\
       \hline
       \emph{generators}&$i_p$&$i_p\alpha_1(2n+2p-2)$&$\tilde{\alpha}_1$
       \\ \hline
    \end{tabular}
 \end{center}
 where $\alpha_k(m)=\Sigma^{m-3}\alpha_k(3)$, $k=1, 2$ and
 \[\alpha_2(2n+1)\in \langle \alpha_1(2n+1),p\iota_{2n+2p-2},\alpha_1(2n+2p-2) \rangle,\]
 the Toda bracket (or called the secondary composition in \cite{TodaBook});
 $\tilde{\alpha}_1$ satisfies the relation
 \begin{equation}\label{eq:alpha1}
    j_p\tilde{\alpha}_1=\alpha_1(2n+2p-1).
 \end{equation}
 
    \begin{proof}
    For each $k$, denote by $\pi_k(X;p)$ the $p$-primary component of $\pi_k(X)$.
     The homotopy groups $\pi_{2n+i}(S^{2n+1};p)$ refer to \citep[Chapter XIII]{TodaBook}. Note that for any $n\geq 1$, the range $i\leq 4p-4$ is the stable range for the group $\pi_{2n+i}(M_{2n+2p-2}(\zp{}))$, which consequently can be computed by the following short exact sequence for $i\leq 4p-4$:
     \[
     \begin{tikzcd}
        \pi_{2n+i}(S^{2n+2p-2};p)\ar[r,tail, "(i_p)_\sharp"]&\pi_{2n+i}(M_{2n+2p-2}(\zp{}))\ar[r,two heads,"(j_p)_\sharp"]&\pi_{2n+i}(S^{2n+2p-1};p),
     \end{tikzcd}\]
     where $i_p=i_p(2n+2p-2),j_p=j_p(2n+2p-2)$ are the canonical inclusion and projection maps, respectively.
    \end{proof}
 \end{lemma}

 \section{Cohomology operations and  modular cohomotopy}\label{sec: oldtools}
 
 In this section let $p$ be a prime and let $r$ be a positive integer; let $G=\zp{r}$ or $\zz{p}$ and let $F_n(G)$ be defined by the $p$-localized \emph{H-fibration sequence} (homotopy fibration sequence of H-spaces and H-maps):
 \begin{equation}\label{fibG}
     F_n(G)\to K_n(G)\xra{\mathcal{P}^1_G}K_{n+2p-2}(\zp{}),
  \end{equation}
 where $\mathcal{P}^1_G$ is the cohomology operation map defined by (\ref{StG}).
 
 \subsection{Certain group extension}\label{sec:EXTG}
 We shall give a unified expression of certain extension of abelian groups induced by  (\ref{fibG}).
 \begin{lemma}\label{lemconn}
   If $n\geq 2p-1$, then there exists a canonical map $
   \alpha_{G}\colon M_n(G)\to F_n(G)$, which is $(n+2p-2)$-connected.
 
  \begin{proof}
    Consider the defined homotopy fibration (\ref{fibG}). There exists a canonical map $\alpha_G\colon M_n(G)\to F_n(G)$ which is at least $n$-connected.
    For simplicity, we write \[H^{n+i}(\zp{r},n;\zp{})=H^{n+i}\big(K_n(\zp{r});\zp{}\big).\]
 
 (1) $G=\zp{r}$. If $n\geq 2p-1$, by Lemma \ref{lemchEM}, the nontrivial cohomology groups $H^{n+i}\big(K_n(\zp{r});\zp{}\big)$ for $i\leq 2p-1$ are listed as follows.
   \begin{center}
 \begin{tabular}{c|c|c|c|c}
 \hline
 $i$&$0$&$1$&$2p-2$&$2p-1$\\
 \hline
 $H^{n+i}(\zp{r},n;\zp{})$&$\iota_n$&$\beta_r\iota_n$&$\mathcal{P}^1\iota_n$&$\mathcal{P}^1\beta_r\iota_n,\beta_1\mathcal{P}^1\iota_n$\\
 \hline
 \end{tabular}
 \end{center}
   where each nonzero generator $x$ represents a direct summand $\zp{}$. Note that when $n=2p-1$, $\iota_n^2=0$ if $p\geq 3$; otherwise $\beta_1\mathcal{P}^1\iota_3=\mathrm{Sq}^3\iota_3=\iota_3^2$.
  By Lemma \ref{lemserre}, the homotopy fibration (\ref{fibG}) induces exact sequences of cohomology groups with coefficients in $\zp{}$:
  \begin{multline*}
     0\to H^{n+i}(\zp{r},n)\xra{}H^{n+i}(F_n(\zp{r}))\to 0, i=0,1,\cdots, 2p-3;\\
 0\to H^{n+2p-2}(\zp{},n+2p-2)\xra{(\mathcal{P}^1)^\ast} H^{n+2p-2}(\zp{r},n)\xra{}H^{n+2p-2}(F_n(\zp{r}))\\
 \to H^{n+2p-1}(\zp{},n+2p-2)\xra{(\mathcal{P}^1)^\ast} H^{n+2p-1}(\zp{r},n)\xra{}H^{n+2p-1}(F_n(\zp{r}))
  \end{multline*}
 Hence $H^{n+i}(F_n(\zp{r});\zp{})$ is isomorphic to $\zp{}$ for $i=0,1$ and $0$ for otherwise $i\leq n+2p-3$. Since $\iota_{n+2p-2}\circ \mathcal{P}^1=\mathcal{P}^1\iota_n$, both representing $\mathcal{P}^1_{\zp{r}}$,  the first $(\mathcal{P}^1)^\ast$ is an isomorphism. The second $(\mathcal{P}^1)^\ast$ is a monomorphism, whose image is $\zp{}\langle \beta_r\mathcal{P}^1\iota_n \rangle$ if $r<\infty$ and otherwise $0$. Thus we get
 \[H^{n+2p-2}(F_n(\zp{r});\zp{})=0,\quad H^{n+2p-1}(F_n(\zp{r});\zp{})\neq 0.\]
 By the universal coefficient theorem we see that the second nontrivial integral reduced homology group of $F_n(\zp{r})$ occurs in dimension $n+2p-1$. Thus the map $\alpha_{\zp{r}}$ is $(n+2p-2)$-connected, by the Whitehead's second theorem  (cf.~\citep[Theorem 6.4.15]{Arkowitzbook}). The proof in the case $G=\zp{r}$ is done.
 
 (2) It is clear that $H^\ast(K_n(\zz{p});\zp{})\cong H^\ast(K_n(\Z);\zp{})$, the proof of the case $G=\zz{p}$ is totally similar and omitted here.
  \end{proof}
 \end{lemma}
 
 \begin{proposition}\label{propEXT}
 Let $X$ be a complex with $\dim(X)\leq n+2p-3$, $n\geq 2p-1$.  Then there is a short exact sequence of abelian groups
 \begin{equation}\label{EXT}
    0\to T_{\Omega \mathcal{P}^1_G}(X)\to \pi^n(X;G)\xra{h^n} H^n(X;G)\to 0,
 \end{equation}
 where $T_{\Omega \mathcal{P}^1_G}(X)=H^{n+2p-3}(X;\zp{})/\mathcal{P}^1_G\big(H^{n-1}(X;G)\big)$. The extension is natural with respect to maps $X\to Y$ and the reduction homomorphisms $\rho^s_r\colon\zp{s}\to \zp{r}$ for $1\leq r\leq s<\infty$; $\rho^r_r=id$.
  \begin{proof}
 Consider the fibration (\ref{fibG}).
 By Lemma \ref{lemconn} and  \citep[Proposition 2.4.6]{Arkowitzbook},  the canonical map $\alpha_G\colon M_n(G)\to F_n(G)$ induces a bijection
 \[[X,M_n(G)]\xra{\cong}[X,F_n(G)]\]
 if $\dim(X)\leq n+2p-3$. Note that $F_n(G)=\Omega F_{n+1}(G)$ is a loop space and hence the bijection above is an isomorphism of abelian groups, by Lemma \ref{lemJamesUP}.
 The proposition then follows by the induced exact sequence:
 \begin{multline*}
    H^{n-1}(X;G)\xra{\mathcal{P}^1_G}H^{n+2p-3}(X;\zp{})\to \pi^n(X;G)\xra{h^n}H^n(X;G)\to 0.
 \end{multline*}
 
 The naturality of the extension is clear.
  \end{proof}
 \end{proposition}
 
 \subsection{Conditions for the splitness of the extension}\label{sec:EXTsplit}
 By \citep[Lemma on page 63]{Hilton65}, the extension (\ref{EXT}) is determined by the multiplication by $p$ on $H^n(X;G)$.
 
 \begin{proposition}\label{propEXT:local}
 If $G=\zz{p},p\geq 2$, then the extension (\ref{EXT}) splits if and only if
  \[\mathcal{P}^1_{G}\big(H^{n-1}(X;G)\big)=\mathcal{P}^1\big(H^{n-1}(X;\zp{})\big)\subseteq H^{n+2p-3}(X;\zp{}).\]
  \begin{proof}
 The proof is totally similar to that of  \citep[Theorem 6.2]{Taylor2012} and we omit the details here.
  \end{proof}
 \end{proposition}
 
 For a prime $p$, let $\pi^n(X;p)$ denote the $p$-primary component of the cohomotopy group $\pi^n(X)$ and by convention, $\pi^n(X;\infty)=$the integral direct summand of $\pi^n(X)$.
 \begin{example}\label{EX:FP-p}
 Let $\mathbb{F}=\mathbb{C},\mathbb{H}$ be the fields of the complex numbers and quaternions, respectively.  Let $p\geq 2$ be a prime and let $\alpha\in H^d(\FP^n;\zp{})$ be a generator, where $d=\dim_\R\mathbb{F}$. Recall that
  \begin{itemize}
     \item[(a)] $H^\ast(\FP^n;\zp{})\cong \Z[\alpha]/(\alpha^{n+1})$.
     \item[(b)] $\mathcal{P}^1(\alpha^m)=\frac{dm}{2}\cdot\alpha^{m+\frac{2p-2}{d}}$ for each $m\geq 1$ (cf.~\citep[4L]{Hatcher}).
  \end{itemize}
 By Proposition \ref{propEXT:local} we compute that
       \begin{enumerate}[(1)]
          \item For  $p\leq \frac{dn}{4}+1$,
          \[\pi^{dn-2p+3}(\FP^{n};p)=\pi^{dn-2p+3}(\FP^{n};\Z_{(p)})=\left\{\begin{array}{ll}
             \Z/p&\text{ if }n\equiv \frac{2p-2}{d}\mod p;\\
             0&\text{otherwise.}
          \end{array}\right.\]
 
          \item  For $p\leq \frac{dn+k}{4}+1,k\geq 1$,
          \[\pi^{dn-2p+k+3}(\mathbb{F}P^{n};\zz{p})\cong H^{dn-2p+k+3}(\mathbb{F}P^{n};\zz{p}),\]
          it follows that
          \[\pi^{dn-2p+k+3}(\mathbb{F}P^{n};\infty)=\left\{\begin{array}{ll}
             \Z& \text{ if $k\leq 2p-3$ is odd};\\
             0&\text{ if $k\leq 2p-3$ is even}.
          \end{array}\right.\]
 
       \end{enumerate}
 
    \end{example}
 
 Taking $p=3$ and comparing with \citep[Theorem 14.2]{West71}, we have
 \begin{corollary}
   $\pi^{4n-3}(\HP^n)\cong \Z/12~\text{or }\Z/24$
    when $n\equiv 1\mod 6$.
 \end{corollary}
 
 Now let's focus on the case $G=\zp{r}$ for some prime $p$.
 If $r\geq 2$, let
 \[K_n(\zp{r})\xra{\rho}  K_n(\zp{r-1})\xra{[p]} K_n(\zp{r})\]
 be the maps corresponding to the reduction $\zp{r}\to \zp{r-1}$ mod $p^{r-1}$ and the multiplication $\cdot p\colon \zp{r-1}\to\zp{r}$, respectively.
 Let $p_K$ be the $p$-{th} power map on $K_n(\zp{r})$, then there holds a homotopy commutative square
 \begin{equation}\label{diag:pK}
    \begin{tikzcd}
    K_n(\zp{r})\ar[r,"\rho"]\ar[d,"{[p]}"]\ar[dr,"p_K"description]&K_n(\zp{r-1})\ar[d,"{[p]}"]\\
    K_n(\zp{r+1})\ar[r,"\rho"]&K_n(\zp{r})\end{tikzcd}
 \end{equation}

 The following two propositions respectively give the conditions for the extension (\ref{EXT})  to split with $G=\zp{r}$, $p\geq 2$ and $r\geq 1$.
 \begin{proposition}\label{propEXTsplit:2}
 Suppose the assumptions in Proposition \ref{propEXT} hold and $G=\z{r}$.
 \begin{enumerate}
        \item If $r\geq 2$, the extension (\ref{EXT}) splits if
        \begin{enumerate}
           \item the homomorphism $[2]\colon H^{n}(X;\z{r-1})\to H^n(X;\z{r})$ is injective and
           \item $\mathrm{Sq}^1\big(H^{n}(X;\z{})\big)\subseteq\sq{\z{r}}\big(H^{n-1}(X;\z{r})\big).$
        \end{enumerate}
     Moreover, if the condition $(a)$ is true, then the extension (\ref{EXT}) splits if and only if $(b)$ holds.
    \item If $r=1$, the extension (\ref{EXT}) splits if
     $\mathrm{Sq}^1\colon H^n(X;\z{})\to H^{n+1}(X;\z{})$ is trivial.
 
 \end{enumerate}
 
 \begin{proof}
 $(1)$
 Consider the following homotopy fibration diagram
 \begin{equation}\label{diagpsi}
 \begin{tikzcd}[column sep=small]
 K_n(\z{})\ar[r,"{[2^{r-1}]}"]\ar[d,"\mathrm{Sq}^1"]&K_n(\z{r})\ar[r,"\rho"]\ar[d,dashed,"\psi"]\ar[dr,dotted,"2_K"description]&K_n(\z{r-1})\ar[r,"\delta"]\ar[d,"{[2]}"]\ar[dr,dotted,"0"description]&K_{n+1}(\z{})\ar[d,"\mathrm{Sq}^1"]\\
 K_{n+1}(\z{})\ar[r,"\partial"]&F_n(\z{r})\ar[r,"\imath"]&K_n(\z{r})\ar[r,"\sq{\z{r}}"]&K_{n+2}(\z{})
       \end{tikzcd}
 \end{equation}
 where $\delta=\mathrm{Sq}^1_{\z{r-1}}$.
 We \textbf{claim} that after choosing suitably, the map $\psi$ satisfies  the relations
    \[ \imath\psi=2_K,\quad\psi\imath=2_F,\]
    where $2_F$ is the $2$-th power map on $F_n(\z{r})$.
 As the dashed arrow indicated, the first equality trivially holds by the commutativity of the middle square.
 The map $\imath\colon F_n(\z{r})\to K_n(\z{r})$ induces an  commutative exact sequence diagram
 \[\begin{tikzcd}[column sep=small]
     H^{n+1}(K_n(\z{r});\z{})\ar[r,tail,"\partial_\sharp"]\ar[d,"\imath^\ast","\cong"swap]&{[K_n(\z{r}),F_n(\z{r})]}\ar[r,"\iota_\sharp"]\ar[d,"\imath^\sharp"]& {[K_n(\z{r}),K_n(\z{r})]}\ar[d,"\imath^\sharp"]\\
     H^{n+1}(F_n(\z{r});\z{})\ar[r,tail,"\partial_\sharp"]&{[F_n(\z{r}),F_n(\z{r})]}\ar[r,"\iota_\sharp"]&{[F_n(\z{r}),K_n(\z{r})]}
 \end{tikzcd}\]
 By Lemma \ref{lemserre}, the first vertical $\imath^\ast$ is an isomorphism.
 The upper exact sequence implies that up to homotopy there are exactly two maps $\psi$ satisfying $\imath\psi=2_K$:
 \[\psi_1=\psi\quad\text{and}\quad \psi_2=\psi+\partial\varepsilon,\]
  where  $\varepsilon\in H^{n+1}(K_n(\z{r});\z{})\cong\z{}$ is a generator.
  The second commutative square  and equation then implies that
 \begin{equation*}
    \psi_1\imath=2_F \quad\text{or} \quad\psi_1\imath=2_F+\partial\varepsilon',
 \end{equation*}
 where $\varepsilon'=\imath^\ast(\varepsilon)\in H^{n+1}(F_n(\zp{r});\z{})\cong \z{}$ is a generator.
 
 If $\psi_1\imath=2_F$, the \textbf{claim} is proved; if not,  replacing $\psi_1$ by $\psi_2=\psi+\partial\varepsilon$, then \[\psi_2\imath=2_F+\partial(\varepsilon'+\varepsilon\imath)=2_F+2\partial \varepsilon'=2_F.\]
 The proof of the \textbf{claim} is completed.
 
 There is a commutative exact sequence diagram induced by (\ref{diagpsi}):
    \[\begin{tikzcd}[column sep=small]
       H^{n-1}(X;\z{r-1})\ar[r,"\delta"]\ar[d,"{[2]}"]&H^{n}(X;\z{})\ar[d,"\mathrm{Sq}^1"]\ar[r,"{[2^{r-1}]}"]&H^n(X;\z{r})\ar[r,"\rho"]\ar[d,"\psi_\sharp"]\ar[dr,dotted,"2"description]&H^{n}(X;\z{r-1})\ar[d,"{[2]}"]\\
             H^{n-1}(X;\z{r})\ar[r,"\sq{\z{r}}"]&H^{n+1}(X;\z{})\ar[r]&\pi^n(X;\z{r})\ar[r,two heads, "h^n"]&H^n(X;\z{r})
    \end{tikzcd}\]
 Given $\gamma\in {}_2H^n(X;\z{r})$, $0=2\gamma=[2]\rho(\gamma)$ is equivalent to $\rho(\gamma)=0$, by the condition $(a)$.
 There exists $\gamma'\in H^{n}(X;\z{})$ satisfying $[2^{r-1}](\gamma')=\gamma$ and hence for any $\bar{\gamma}\in \pi^n(X;\z{r})$ mapping to $\gamma$ by $h^n$, by the relation $\psi\imath=2_F$ and the middle commutative square, we obtain
 \[2\bar{\gamma}=\psi(\gamma)=\mathrm{Sq}^1(\gamma')\in \coker(\sq{\z{r}})=T_{\Omega \sq{\z{r}}}(X).\]
 Thus by \citep[Lemma on page 63]{Hilton65}, the extension (\ref{EXT}) splits if, in addition,
 \[\mathrm{Sq}^1\big(H^n(X;\z{})\big)\subseteq \sq{\z{r}}\big(H^{n-1}(X;\z{r})\big).\]
 
 $(2)$ Consider the following homotopy fibration diagram
 \[\begin{tikzcd}
 K_n(\Z/4)\ar[r,"\rho"]\ar[d,"\mathrm{Sq}^1_{\Z/4}"]&K_n(\z{})\ar[r,"\mathrm{Sq}^1"]\ar[d,"\psi'"]&K_{n+1}(\z{})\ar[r,"{[2]}"]\ar[d,"0"]&K_{n+1}(\Z/4)\ar[d,"\mathrm{Sq}^1_{\Z/4}"]\\
 K_{n+1}(\z{})\ar[r,"\partial"]&F_n(\z{})\ar[r,"\imath"]&K_n(\z{})\ar[r,"\sq{}"]&K_{n+2}(\z{})
 \end{tikzcd}\]
 We argue in a similar way to that of the proof of the \textbf{claim} above that the map $\psi'\colon K_n(\Z/2)\to F_n(\z{})$ satisfies the relations:
 \[\imath\psi'=2_K=0,\quad\psi'\imath=2_F.\]
 Consider the following induced commutative diagram:
 \[\begin{tikzcd}[column sep=small]
    H^{n}(X;\z{})\ar[r,"{[2]}"]\ar[d,"0"]&H^{n}(X;\Z/4)\ar[d,"\mathrm{Sq}^1_{\Z/4}"]\ar[r,"\rho"]&H^n(X;\z{})\ar[dr,"2"description, dotted]\ar[r,"\mathrm{Sq}^1"]\ar[d,"\psi'_\sharp"]&H^{n+1}(X;\z{})\ar[d,"0"]\\
    H^{n+1}(X;\z{})\ar[r,"\sq{}"]&H^{n+1}(X;\z{})\ar[r]&\pi^n(X;\z{})\ar[r,two heads,"h^n"]&H^n(X;\z{})
 \end{tikzcd}\]
 Given $\gamma\in {}_2H^n(X;\z{})=H^n(X;\z{})$, by assumption and the exactness of the upper row, there exists $\gamma'\in H^n(X;\Z/4)$ such that $\rho (\gamma')=\gamma$. For any $\bar{\gamma}\in \pi^n(X;\z{})$ we have
 \[\mathrm{Sq}^1_{\Z/4}(\gamma')=\psi'(\gamma)=2\bar{\gamma}=0.\]
 Thus  the extension (\ref{EXT}) splits, by \citep[Lemma on page 63]{Hilton65}.
  \end{proof}
 \end{proposition}

 \begin{proposition}\label{propEXTsplit:p}
 Suppose the assumptions in Proposition \ref{propEXT} hold and let $p\geq 3$, $1\leq r<\infty$. The extension (\ref{EXT}) for $G=\zp{r}$ splits if one of the following conditions holds:
 \begin{enumerate}
    \item $r\geq 2$ and the homomorphism $[p]\colon H^n(X;\Z/p^{r-1})\to H^n(X;\zp{r})$ is injective.
    \item $r=1$.
 \end{enumerate}
 
 \begin{proof}
 $(1)$ There is a homotopy fibration diagram
    \[\begin{tikzcd}[column sep=small]
       K_n(\zp{})\ar[r,"{[p^{r-1}]}"]\ar[d,"0"]&K_n(\zp{r})\ar[r,"\rho"]\ar[d,"\phi"]\ar[dr,dotted,"p_K"description]&K_n(\zp{r-1})\ar[r,"\delta"]\ar[d,"{[p]}"]&K_{n+1}(\zp{})\ar[d,"0"]\\
       K_{n+2p-3}(\zp{})\ar[r,"\partial"]&F_n(\zp{r})\ar[r,"\imath"]&K_n(\zp{r})\ar[r,"\mathcal{P}^1_{\zp{r}}"]&K_{n+2p-2}(\zp{})
    \end{tikzcd}\]
 Consider the following commutative exact sequence diagram
 \[\begin{tikzcd}[column sep=small]
    0\ar[r]&{[K(\zp{r},n),F_n(\zp{r})]}\ar[d,"\imath^\sharp"]\ar[r,"\iota_\sharp"]&{[K_n(\zp{r}),K(\zp{r},n)]}\ar[d,"\imath^\sharp"]\\
    0\ar[r]&{[F_n(\zp{r}),F_n(\zp{r})]}\ar[r,"\iota_\sharp"]&{[F_n(\zp{r}),K_n(\zp{r})]}
 \end{tikzcd}\]
 where $H^{n+2p-3}(F_n(\zp{r});\Z/p)=0$ is due to Lemma \ref{lemserre}. It follows that the map $\phi$ satisfying $\imath\phi=p_K$ is unique and satisfies the relation $\phi\imath=p_F$
 
 The remainder of the proof of $(1)$ is similar to that of Proposition \ref{propEXTsplit:2}, by analyzing the induced commutative exact sequence diagram:
    \begin{equation*}
       \begin{tikzcd}[column sep=small]
          H^{n-1}(X;\zp{r-1})\ar[r]\ar[d,"{[p]}"]&H^{n}(X;\zp{})\ar[d,"0"]\ar[r,"{[p^{r-1}]}"]&H^n(X;\zp{r})\ar[dr,dotted,"p" description]\ar[r,"{\rho}"]\ar[d,"\phi_\sharp"]&H^{n}(X;\zp{r-1})\ar[d,"{[p]}"]\\
          H^{n-1}(X;\zp{r})\ar[r,"\pzpr"]&H^{n+2p-3}(X;\zp{})\ar[r]&\pi^n(X;\zp{r})\ar[r,two heads, "h^n"]&H^n(X;\zp{r})
       \end{tikzcd}
    \end{equation*}
 
  $(2)$ There is a homotopy fibration diagram
  \[\begin{tikzcd}
    K_n(\zp{2})\ar[r,"\rho"]\ar[d,"0"]&K_n(\zp{})\ar[r,"\delta"]\ar[d,"\phi'"]&K_{n+1}(\zp{})\ar[r,"{[p]}"]\ar[d,"0"]&K_{n+1}(\zp{2})\ar[d,"0"]\\
    K_{n+2p-3}(\zp{})\ar[r,"\partial"]&F_n(\zp{})\ar[r,"\imath"]&K_n(\zp{})\ar[r,"\mathcal{P}^1"]&K_{n+2p-2}(\zp{})
    \end{tikzcd}\]
  There hold relations
   \[\imath\phi'=p_K=0, \quad\phi'\imath=p_F=0\]
 and the extension (\ref{EXT}) splits.
 \end{proof}
 \end{proposition}
 
 \begin{example}\label{ex:Moore}
    Let $n\geq 3,r,s\geq 1$.
  \begin{enumerate}
     \item Write $M^n_{2^s}=M_n(\z{s})$. There is a short exact sequence
     \begin{equation}\label{sesMoore}
        0\to H^{n+1}(M^n_{2^s};\z{})\to \pi^n(M^n_{2^s};\z{r})\to H^n(M^n_{2^s};\z{r})\to 0,
     \end{equation}
    which splits if and only if $r\geq 2$.
    \item  Generally, for any an $(n-1)$-connected $(n+1)$-dimensional finite complex $X$, the short exact sequence
    \begin{equation}\label{ses:An2}
     0\to H^{n+1}(X;\z{})\to \pi^n(X;\z{r})\xra{h^n} H^n(X;\z{r})\to 0
    \end{equation}
  splits if and only if $H_n(X)$ has no direct summand $\z{}$.
  \end{enumerate}
 \begin{proof}
 $(1)$ Since $M^n_{2^s}$ is $(n-1)$-connected, $T_{\Omega \sq{\z{r}}}(M^n_{2^s})\cong H^{n+1}(M^n_{2^s};\z{})$. The condition $(a)$ in Proposition \ref{propEXTsplit:2} $(1)$ is true. Thus the extension (\ref{sesMoore}) splits if and only if
 \[0=\mathrm{Sq}^1\colon H^n(M^n_{2^s};\z{})\to H^{n+1}(M^n_{2^s};\z{}),\]
 which holds if and only if $s\geq 2$, by the Bockstein exact sequence associated to the short exact sequence $0\to \z{}\to \Z/4\to \z{}\to 0$. The statement for the mod $2^s$ Moore space follows by the Spanier-Whitehead duality:
 \[[M^n_{2^s},M^n_{2^r}]\cong [M^n_{2^r},M^n_{2^s}].\]
 
 $(2)$ It is well-known that an $(n-1)$-connected $(n+1)$-dimensional finite complex $X$ has the homotopy type of a wedge of spheres and mod $p^r$ Moore spaces for different primes $p\geq 2$ and positive integers $r$; that is, there is a homotopy equivalence
 \[X\simeq \bigvee_{i}S^n_i\vee \bigvee_{j}S^{n+1}_j\vee \bigvee_{p,r}M_n(\zp{r}),\]
 where $S^{n}_i=S^n$ and $S^{n+1}_j=S^{n+1}$.
 By Theorem \ref{thmPeterson} (\ref{thmACH}), it suffices to prove $(2)$ for the indecomposable spaces $X=S^n,S^{n+1}$ and $M_n(\zp{r})$. The $X=S^n,S^{n+1}$ or $M_n(\zp{r})$ with $p\geq 3$ cases are trivial and the equivalent condition in $(1)$ is obviously characterized by the condition that $H_n(X)$ has no direct summand $\z{}$. The proof of $(2)$ completes.
   \end{proof}
 \end{example}

 \begin{proof}[Proof of Theorem \ref{mainthm1}]
    Theorem \ref{mainthm1} consists of Proposition \ref{propEXT}, \ref{propEXT:local}, \ref{propEXTsplit:2}, \ref{propEXTsplit:p}.
 \end{proof}

 \section{Unstable homotopy theory and modular cohomotopy}\label{sec:unstHtp}
 
 In this section we combine some classic unstable homotopy theory of Moore space to study the modular cohomotopy groups.
 Let $p$ be a prime.
 It is well-known that $S^{2n-1}_{(p)}$ is an H-space for $p\geq 3, n\geq 1$ and $S^{n-1}_{(2)}$ is an H-space if and only if $n=1,2,4,8$ \cite{Adams61}.
 If $p\geq 5,n\geq 1$ (or $p\geq 2$, $n=1,2$), $S^{2n-1}_{(p)}$ is homotopy associative; if $p\geq 5$ (or $p\geq 3,n\geq 2$), $S^{2n-1}_{(p)}$ is homotopy commutative, see \citep[page 465, Exercises]{Neisendorferbook}.
 
 \subsection{$(2n+1)$-connected cover of $S^{2n+1}$}\label{sec:mod-p}
 Consider the following $p$-localized homotopy fibration sequence
 \begin{align*}
    &K_{2n}(\zz{p})\xra{\partial_{(p)}}\s{2n+1}_{(p)}\xra{\pi_{(p)}}S^{2n+1}_{(p)}\xra{\iota_{(p)}}K_{2n+1}(\zz{p})
 \end{align*}
 where  $\iota=\iota_{2n+1}$ represents a generator of $H^{2n+1}(S^{2n+1})$.
 It is clear that for each $p\geq 3$, the map $\iota_{(p)}$ is an H-map and hence the above sequence is an H-fibration sequence.
 
 \begin{lemma}\label{lem:oc}
  For each prime $p\geq 2$,  $\iota_{(p)}=\Omega \alpha$ for some  $\alpha\in H^{4}(\B S^{3}_{(p)};\zz{p})$, hence $\s{3}_{(p)}$ inherits the H-space structure of $S^{3}_{(p)}$.
    \begin{proof}
    Since $S^3$ has the homotopy type of a loop space, so does $\s{3}$; hence the classifying space $\B S^{3}_{(p)}$ exists. The evaluation $e\colon \Sigma\Omega \B S^{3}_{(p)}\to\B S^{3}_{(p)}$ is $7$-connected and hence the loop map
       \[ [\B S^{3}_{(p)},K_{4}(\zz{p})]\xra{e^\sharp}[\Sigma S^{3}_{(p)},K_{4}(\zz{p})]\cong [S^{3}_{(p)},K_{3}(\zz{p})]\]
       is surjective, by \citep[Proposition 2.4.13]{Arkowitzbook}.
    \end{proof}
 \end{lemma}

 \begin{proposition}\label{propconn:oc}
    Let $p$ be a prime and let $n\geq 1$.
 \begin{enumerate}
 \item There exists a map
    $\alpha\colon M_{2n+2p-2}(\zp{})\to \s{2n+1}_{(p)}$ which is $N$-connected with
    \[N=\left\{\begin{array}{ll}
       2n+4p-3&\text{for }p\geq 3 \text{ or }p=2,n=1;\\
       2n+3&\text{for }p=2,n\geq 2.
    \end{array}\right.\]
 \item If $p\geq 3$ or $p=2,n=1$, then  there exists a map
 $M_{2n+2p-3}(\zp{})\to \Omega\s{2n+1}_{(p)}$ which is $(2n+4p-4)$-connected.
 
 \item For each $p\geq 2$, there exists a map
 $M_{2p+1}(\zp{})\to \B\s{3}_{(p)}$ which is $4p$-connected.
 \end{enumerate}

 \begin{proof}
  For simplicity we write $M_{k}=M_{k}(\Z/p)$ for $k\geq 1$ in the proof.  The $n=1$ cases in $(1)$ and $(2)$ refer to \citep[Corollary 4.6.2 and  Proposition 4.6.4]{Neisendorferbook}, respectively.
 Suppose $n>1$.
 
 $(1)$ By Lemma \ref{lemToda} we have an isomorphism
 \[[M_{2n+2p-2},S^{2n+1}]\xra[\cong]{i_p^\sharp}[S^{2n+2p-2},S^{2n+1}]\cong\zp{}\langle \alpha_1(2n+1)\rangle.\]
 Let $\bar{\alpha}_1$ be the inverse image of $\alpha_1(2n+1)$ under $i_p^\sharp$. There exists a canonical map $\alpha\colon M_{2n+2p-2}\to\s{2n+1}_{(p)}$ satisfying $\bar{\alpha}_1=\pi_{(p)}\circ \alpha$.
 Since $\pi_{(p)}$ induces an isomorphism of homotopy groups in dimensions greater than $2n+1$,  $\alpha$ is $N$-connected if and only if $\bar{\alpha}_1$ is $N$-connected.

 $(i)$ Let $p\geq 3$. Using the elements defined in Lemma \ref{lemToda}, there hold relation equalities:
 \begin{enumerate}[$(a)$]
    \item $\bar{\alpha}_1\circ i_p=\alpha_1(2n+1)$, by definition.
    \item $\bar{\alpha}_1\circ (i_p\alpha_1(2n+2p-2))=\alpha_1(2n+1)\circ \alpha_1(2n+2p-2)$.
    \item $\bar{\alpha}_1\circ \tilde{\alpha}_1=-\alpha_2$, due to \citep[Proposition 4.2]{Toda59} by taking \[\langle \alpha,\beta,\gamma\rangle=\langle \alpha_1(2n+1),p\iota_{2n+2p-2},\alpha_1(2n+2p-2)\rangle.\]
 \end{enumerate}
 It follows that $\bar{\alpha}_1$ induces an isomorphism of homotopy groups in dimensions  $\leq 2n+4p-4$.
 Observe that $\pi_{2n+4p-3}(M_{2n+2p-2})\neq 0$, while $\pi_{2n+4p-3}(S^{2n+1};p)=0$ by \citep[Theorem 13.4]{TodaBook}. Thus $\bar{\alpha}_1$, or equivalently $\alpha$ is $(2n+4p-3)$-connected if $p\geq 3$.
 
 $(ii)$ If $p=2$, the map $\bar{\alpha}_1=\bar{\eta}_1\colon M_{2n+2}\to S^{2n+1}$ satisfies the relation $\bar{\eta}_1i_2=\eta$.
  By \cite{Baues91} we have $\pi_{2n+3}(M_{2n+2})\cong\z{}\langle i_2\eta\rangle$ and $\pi_{2n+4}(M_{2n+2})\cong \Z/4\langle \tilde{\eta}_1\rangle$, where $\tilde{\eta}_1$ satisfies $j_2\tilde{\eta}_1=\eta$.
  Recall that $\pi_{2n+4}(S^{2n+1};2)\cong \Z/8\langle \nu\rangle$ if $n>1$.
 Hence $\bar{\alpha}_1$ induces an isomorphism of homotopy groups in dimensions $\leq 2n+3$ and therefore $\bar{\alpha}_1$ is $(2n+3)$-connected. This completes the proof of $(1)$.
 
 $(2)$ The loop suspension map  $E\colon M_{2n+2p-3}\to \Omega M_{2n+2p-2}$ is $(4n+4p-7)$-connected.
 Consider the adjoint map $\alpha^\circ=(\Omega \alpha)\circ E$:
 \[ M_{2n+2p-3}\xra{E_1} \Omega M_{2n+2p-2}\to \Omega \s{2n+1}_{(p)}.\]
 Since $(\Omega\alpha)$ is $(2n+4p-4)$-connected and $2n+4p-4\leq 4n+4p-7$ for $n>1$, we see that $\alpha^\circ$ is $(2n+4p-4)$-connected. 
 
 (3)  By \citep[(3.2)]{BM58}, the evaluation map $e\colon \Sigma\s{3}_{(p)}\to \B\s{3}_{(p)}$ is $(4p+1)$-connected.  Then $(1)$ implies that the adjoint map
 \[\alpha^\diamond\colon M_{2p+1}\xra{\Sigma \alpha}\Sigma\s{3}_{(p)}\xra{e}\B\s{3}_{(p)}\]
 is $4p$-connected.
 \end{proof}
 \end{proposition}

 \begin{proof}[Proof of Theorem \ref{mainthm:oc}]
  By Proposition \ref{propconn:oc} and Lemma \ref{lemJamesUP}, if $p\geq 3,n\geq 1$ or $p=2,n=1$, there are induced bijections
  \begin{equation}\label{Eq:cc-1}
 \pi^{2n+2p-2}(X;\zp{})\xra{1:1}[X,\s{2n+1}_{(p)}]
  \end{equation}
 for $\dim(X)\leq 2n+4p-4$ and
 \begin{equation}\label{Eq:cc-2}
    \pi^{2n+2p-3}(X;\zp{})\xra{1:1}[X,\Omega\s{2n+1}_{(p)}]
 \end{equation}
 for $\dim(X)\leq 2n+4p-5$.
 
 (1) If $p\geq 5$ or $n=1, p\geq 2$, $S^{2n+1}_{(p)}$ and $\s{2n+1}_{(p)}$ are homotopy associative H-spaces. By Lemma \ref{lemJamesUP}, the bijection (\ref{Eq:cc-1}) is an isomorphism.
 Similarly the bijection (\ref{Eq:cc-2}) is an isomorphism of abelian groups if $n\geq 2$.
 For $n=1$, this bijection yields  an abelian group structure on $\pi^{2p-1}(X;\zp{})$, and hence the exact sequence in (\ref{mainthm:oc-1}) follows.
 
 If $p\geq 3$, $\pi^{2n+2p-2}(X;\zp{})~(n\geq 1)$ and $\pi^{2n+2p-3}(X;\zp{})~ (n\geq 2)$ are both vector spaces over $\zp{}$, by Theorem \ref{thmPeterson} (\ref{lemsplit:zp}).
 If $n=1$, by \citep[Proposition 8.4]{Theriault2001}, the Anick's space $T^{2p-1}_{\infty}(p)$ is homotopy equivalent to $\Omega\s{3}_{(p)}$.
 By  \citep[Theorem 1.3]{Theriault2001} and its comments, for each $p\geq 3$, $\Omega\s{3}_{(p)}$ has H-space exponent $p$, which means the order of the identity map is $p$. Hence the group $\pi^{2p-1}(X;\zp{})\cong [X,\Omega\s{3}_{(p)}]$ is annihilated by $p$; that is, $\pi^{2p-1}(X;\zp{})$ is a vector space over $\zp{}$.
 
 (2) By Lemma \ref{lemToda}, the map $\B\s{3}_{(p)}\to K_{2p+1}(\zp{})$ representing a generator of $H^{2p+1}(\B\s{3}_{(p)};\zp{})\cong\zp{}$ is $(4p-2)$-connected.  It follows that the induced homomorphism
 \[[X,\B\s{3}_{(p)}]\xra{}H^{2p+1}(X;\zp{})\] is an isomorphism if $\dim(X)\leq 4p-3$.
 
 There is a homotopy fibration diagram induced by the third homotopy commutative square:
 \[\begin{tikzcd}
    K_{2}(\zz{p})\ar[r,"\partial_{(p)}"]\ar[d,equal]&\s{3}_{(p)}\ar[d,"\beta", dashed]\ar[r]&S^{3}_{(p)}\ar[d,"\alpha"]\ar[r,"\iota_{(p)}"]&K_{3}(\zz{p})\ar[d,equal]\\
    K_{2}(\zz{p})\ar[r,"\mathcal{P}^1_{\zz{p}}"]&K_{2p}(\zp{})\ar[r]&F_1(\zz{p})\ar[r]&K_{3}(\zz{p})
 \end{tikzcd}\]
 By Lemma \ref{lem:oc} (1), the above fibration diagram extends from the right by
 \begin{equation}\label{diag:B}
    \begin{tikzcd}
       K_{3}(\zz{p})\ar[r,"\B \partial_{(p)}"]\ar[d,equal]&\B\s{3}_{(p)}\ar[d,"\B\beta"]\\
       K_{3}(\zz{p})\ar[r,"\mathcal{P}^1_{\zz{p}}"]&K_{2p+1}(\zp{})
    \end{tikzcd}
 \end{equation}
 Then the extended exact sequence in (\ref{mainthm:oc-2}) follows at once.
 \end{proof}
 
 \begin{proof}[Proof of Corollary \ref{maincor}]
 (1) If $\dim(X)\leq 6$, the canonical map $M_5(\z{})\to \B\s{3}_{(2)}$ in Proposition \ref{propconn:oc} (3) induces an isomorphism
 $$\pi^5(X;\z{})\cong [X,\B\s{3}_{(2)}].$$
  Then (1) follows by applying Theorem \ref{mainthm:oc}.
 
 (2) The map $\B\beta\colon \B\s{3}_{(2)}\to K_5(\z{})$ defined by the commutative square (\ref{diag:B}) with $(n,p)=(1,2)$ gives a generator of $H^5(\B\s{3}_{(2)};\z{})\cong\z{}$ and is $6$-connected. Thus $[X,\B\s{3}_{(2)}]\to H^5(X;\z{})$ is an isomorphism for $\dim(X)\leq 5$. By Lemma \ref{lem:oc} and \citep[Theorem 1.1]{LT72-B}, the central extension (\ref{EXT:5-cpx}) then follows by (1) and (\ref{diag:B}).
 
 Since $\ker(\mathrm{Sq}^2_{\zz{2}})$ has order $2$,  the central extension (\ref{EXT:5-cpx}) is determined by the two functions $\Gamma, \Phi_2$ given in the statement (2), by \citep[Theorem 3.1]{LT72-B}. We prove the formula for $\Phi_2$ as follows.
 There is a homotopy fibration diagram:
 \[\begin{tikzcd}
    K_{2}(\zz{2})\ar[d,equal]\ar[r,"\rho"]&K_{2}(\z{})\ar[d,"\partial"]\ar[r]&K_3(\zz{2})\ar[d,"\phi",dashed]\ar[r,"2"]&K_3(\zz{2})\ar[d,equal]\ar[r,"\rho"]&K_3(\z{})\ar[d,"\B\partial"]\\
    K_2(\zz{2})\ar[r,"\partial_{(2)}"]&\s{3}_{(2)}\ar[r]&S^3_{(2)}\ar[r,"\iota_{(2)}"]&K_3(\zz{2})\ar[r,"\B\partial_{(2)}"]&\B \s{3}_{(2)}
 \end{tikzcd}\]
 where $\rho$ is the mod $2$ reduction.
 By a similar argument to that of the \textbf{claim} of the proof of Proposition \ref{propEXTsplit:2} (1) we deduce that there is an appropriate $\phi_k$ satisfying the relations
 \[\iota_{(2)}\circ\phi=2, \qquad\phi\circ \iota_{(2)}=2.\]
 Consider the induced commutative diagram of exact sequences:
 \[\begin{tikzcd}
    H^2(X;\zz{2})\ar[r]\ar[d,equal]&H^2(X;\z{})\ar[d,"\partial_{2\sharp}"]\ar[r,"\delta"]&H^3(X;\zz{2})\ar[r,"{[2]}"]\ar[d,"\phi_{\sharp}"]&H^3(X;\zz{2})\ar[d,equal]\\
    H^2(X;\zz{2})\ar[r,"\partial_{\sharp}"]&\pi^4(X;\z{})\ar[r]&\pi^3(X;\zz{2})\ar[r,"h^3"]&H^3(X;\zz{2})
 \end{tikzcd}\]
 Then given any $\alpha=h^3(\bar{\alpha})=\ker(\mathrm{Sq}^2_{\zz{2}})$ with $\bar{\alpha}\in \pi^3(X;\zz{2})$, $2\cdot \alpha=0$ implies that there exists $\alpha'\in H^2(X;\z{})$ such that $\delta(\alpha')=\alpha$ and we obtain
  \[\bar{\alpha}^2=\phi(\alpha)=\partial_2(\alpha')\in \frac{\pi^4(X;\z{})}{\partial_\sharp(H^2(X;\zz{2}))}.\]
 \end{proof}

 \subsection{The double suspension}\label{sec:doubsus}
 Consider the $p$-localized H-fibration sequence constructed by Gray \cite[Theorem 9]{Gray88} for each prime $p\geq 5$:
 \begin{equation}\label{fib:Cn}
    C(n)_{(p)}\xra{\Omega v_n} S^{2n-1}_{(p)}\xra{E^2_{(p)}}\Omega^2 S^{2n+1}_{(p)}\xra{v_n}\B C(n)_{(p)}
 \end{equation}

 \begin{lemma}\label{lem:Cn}
  Let $p\geq 3$ be a prime and let $n\geq 1$.
  \begin{enumerate}
     \item There exists a canonical map $\beta\colon M_{2pn-3}(\zp{})\to C(n)_{(p)}$ which is $(2pn+2n-5)$-connected.
     \item There exists a canonical map $M_{2pn-2}(\zp{})\to \B C(n)_{(p)}$ which is $(2pn+2n-4)$-connected.
  \end{enumerate}
  \begin{proof}
  $(1)$ From \citep[page 309]{Neisendorferbook}, for each $p\geq 3,n\geq 1$, we have
    \begin{equation*}
       H_k(C(n)_{(p)})\cong\left\{\begin{array}{ll}
          \zp{}&\text{ for }k=2pn-3;\\
          0&\text{ otherwise for }k\leq 2pn+2n-5.
       \end{array}\right.
    \end{equation*}
  It follows that the canonical inclusion map $\beta\colon M_{2pn-3}(\zp{})\to C(n)_{(p)}$ is $(2pn+2n-5)$-connected.  
 
  $(2)$ The evaluation map $e\colon \Sigma C(n)_{(p)}=\Sigma\Omega\B C(n)_{(p)}\to \B C(n)_{(p)}$ is $(4pn-5)$-connected,  hence the adjoint map $\beta^{\diamond}=e\circ \Sigma(\beta)$ is $(2pn+2n-4)$-connected and the proof completes.
  \end{proof}
 \end{lemma}
 
 \begin{proof}[Proof of Theorem \ref{mainthm:Cn}]
 The H-fibration sequence (\ref{fib:Cn}) induces an exact sequence
  \begin{multline*}
    [X,\Omega S^{2n-1}_{(p)}]\xra{(\Omega E^2_{(p)})_\sharp}[X,\Omega^3S^{2n+1}_{(p)}]\xra{}  [X,C(n)_{(p)}]\\
    \xra{(\Omega v_n)_\sharp}[X,S^{2n-1}_{(p)}]\xra{(E^2_{(p)})_\sharp}[X,\Omega^2S^{2n+1}_{(p)}]
     \xra{(v_n)_\sharp}[X,\B C(n)_{(p)}].
  \end{multline*}
  By Lemma \ref{lem:Cn}, \ref{lemJamesUP}, there are natural isomorphisms of abelian groups
  \begin{align*}
   &\pi^{2pn-3}(X;\zp{})\xra{\cong}[X,C(n)_{(p)}], \text{ if }\dim(X)\leq 2pn+2n-6; \\
    &\pi^{2pn-2}(X;\zp{})\xra{\cong}[X,\B C(n)_{(p)}],\text{ if }\dim(X)\leq 2pn+2n-5.
  \end{align*}
 Since $C(n)_{(p)}$ has H-space exponent $p$ \cite[cf.~][Corollary 1.5]{CMN79-2} or by Theorem \ref{thmPeterson} (\ref{lemsplit:zp}), $\pi^{2pn-3}(X;\zp{})$ is a vector space over $\zp{}$ and the proof is completed.
 \end{proof}

 \subsection{Anick's fibration}\label{sec:Anick}
 Let $p$ is an odd prime in this subsection.  When $p\geq 5$, $n,r\geq 1$, Anick \cite{AnickBook} constructed the following homotopy fibration (\ref{fib:Anick}), which was extended to $p=3$ by Gray and Theriault \cite{GT10}:
 \begin{lemma}[cf.~ \cite{GT10}]\label{lemTh}
 For integers $n,r\geq 1$, there exists a space $\T{2n-1}$ and an H-fibration
 sequence
 \begin{equation}\label{fib:Anick}
    \Omega^2S^{2n+1}_{(p)}\xra{\phi_r}S^{2n-1}_{(p)}\xra{\imath_r} \T{2n-1}\xra{\pi_\infty} \Omega S^{2n+1}_{(p)}.
 \end{equation}
 with the following two properties:
  \[\phi_r\circ E^2_{(p)}\simeq p^r\qquad  E^2_{(p)}\circ \phi_r\simeq \Omega^2 p^r,\] where  $E^2_{(p)}$
 is the double suspension localized at $p$.
 
 Moreover, there is a coalgebra isomorphism
 \[H_\ast(\T{2n-1};\zp{})\cong\zp{}[v_{2n}]\otimes \Lambda(u_{2n-1}),\]
 where subscripts indicate dimensions and $\beta^r(v_{2n})=u_{2n-1}$ under the $r$-{th} Bockstein.
 
 \end{lemma}
 The space $\T{2n-1}$ is called the \emph{Anick's space} and the fibration (\ref{fib:Anick}) is the called the \emph{Anick's fibration}.
 
 \begin{lemma}\label{lemAnick:conn}
 Let $p$ be an odd prime and let $n,r\geq 1$.
 \begin{enumerate}
    \item There exists a canonical map $\imath_M\colon M_{2n-1}(\zp{r})\to \T{2n-1}$ which is $(4n-2)$-connected.
    \item If $n\geq 2$, there is a canonical map $M_{2n-2}(\zp{r})\to \Omega\T{2n-1}$ which is $(4n-5)$-connected.
 \end{enumerate}
 
 \begin{proof}
 $(1)$ is a direct result of the coalgebra structure of $H_\ast(\T{2n-1};\zp{})$.
 
 $(2)$ Write $M=M_{2n-2}(\zp{r})$ for short. The adjoint map $\imath_M^\circ$ is the composition $(\Omega \imath_M)E_M$, where $E_M\colon M\to \Omega \Sigma M$ is the loop suspension.
 By $(1)$ the loop map $(\Omega \imath_M)$ is $(4n-3)$-connected.
 Since $E_M$ is $(4n-5)$-connected, $\imath_M^\circ$ is $(4n-5)$-connected and the proof completes.
 \end{proof}
 \end{lemma}
 
 \begin{remark}\label{rmk:Anick}
 \begin{enumerate}
    \item Let $i_{(p)}\colon S^{2n-1}_{(p)}\xra{}M_{2n-1}(\zp{r})$ be the canonical inclusion map localized at $p$, then we have \[\imath_{r}=\imath_M\circ i_{(p)}.\]
    \item Let $S^{2n-1}\{p^r\}$ be the homotopy fibre of the $p^r$-th power map of $S^{2n-1}_{(p)}$. By \citep[page 472]{Neisendorferbook}, for each $n\geq 2$, the canonical map $M_{2n-2}(\zp{r})\to S^{2n-1}\{p^r\}$ is $(4n-5)$-connected. Together with  Proposition \ref{lemAnick:conn} $(2)$, we obtain isomorphisms for $\dim(X)\leq 4n-6$:
     \[\pi^{2n-2}(X;\zp{r})\xra{\cong}[X,S^{2n-1}\{p^r\}]\xra{\cong}[X,\Omega\T{2n-1}].\]
 
 \end{enumerate}
 \end{remark}
 
 Since $T^1_{\infty}(p^r)\simeq K_1(\zp{r})\times\Omega \s{3}_{(p)}$ \citep[Lemma 5.1]{Theriault2001-2}, the group $[X,\Omega\s{3}_{(p)}]$ has been discussed, we may assume that $n\geq 2$.
 
 \begin{proof}[Proof of Theorem \ref{mainthm:GUCT}]
 For any prime $p\geq 5$ and any complex $X$,  $[X,\T{2n-1}]$ is an abelian group since $\T{2n-1}$ is a homotopy commutative H-space \citep[Theorem 1.2]{Theriault2001}. The Anick's fibration (\ref{fib:Anick}) induces an exact sequence of abelian groups
 \begin{multline*}
    [X,\Omega^2S^{2n+1}_{(p)}]\xra{(\phi_r)_\sharp} [X,S^{2n-1}_{(p)}]\xra{(\imath_r)_\sharp}[X,\T{2n-1}]\xra{(\pi_\infty)_\sharp}[X,\Omega S^{2n+1}].
 \end{multline*}
 
 Consider the following bijection for complexes $X$ of dimension $\leq 4n-3$:
 \begin{equation}\label{bij:Anick}
    \pi^{2n-1}(X;\zp{r})\xra{}[X,\T{2n-1}].
 \end{equation}
 If $\dim(X)\leq 4n-4$, by Lemma \ref{lemJamesUP} this bijection is an isomorphism of abelian groups.
 If $\dim(X)=4n-3$, the bijection (\ref{bij:Anick}) provides $\pi^{2n-1}(X;\zp{r})$ an abelian group structure.
 
 Consider the commutative square of abelian groups:
 \[\begin{tikzcd}
    {[X,S^{2n-1}_{(p)}]}\ar[r," p^r\cdot"]\ar[d,"(E^2_{(p)})_\sharp"]&{[X,S^{2n-1}_{(p)}]}\ar[d,equal]\\
    {[X,\Omega^2S^{2n+1}_{(p)}]}\ar[r,"(\phi_r)_\sharp"]&\pi^{2n-1}(X;\zz{p})
 \end{tikzcd}
 \]
  Since $E^2_{{(p)}}$ is $(2pn-3)$-connected, $(E^2_{(p)})_\sharp$ is an isomorphism and hence
 \begin{align*}
    T_{\phi_r}(X)&\cong\coker\big(\pi^{2n-1}(X;\zz{p})\xra{p^r}\pi^{2n-1}(X;\zz{p})\big)\\
    &\cong \pi^{2n-1}(X;\zz{p})\otimes\zp{r}
 \end{align*}
 and the exact sequence in Theorem \ref{mainthm:GUCT} is proved.
 
 By \citep[Theorem 1.3]{Theriault2001}, $\T{2n-1}$ has H-space exponent $p^r$ for each $p\geq 5$. Hence $\pi^{2n-1}(X;\zp{r})\cong [X,\T{2n-1}]$ is annihilated by $p^r$; especially, $\pi^{2n-1}(X;\zp{})$ is a vector space over $\zp{}$, $T_{\phi_1}(X)$ is a direct summand.
 The proof of $T_{\phi_r}(X)$ is a direct summand for $r\geq 2$ is similar to that of \citep[Lemma 7.3]{Peterson56-2}; the details are given as follows.
 Consider the commutative diagram induced by coefficient homomorphism $\rho\colon\zp{r}\to \zp{}$:
 \[\begin{tikzcd}
   0\ar[r]& \pi^{2n-1}(X;\zz{p})\otimes\zp{r}\ar[r,"i_r"]\ar[d,"1\otimes\rho"]&\pi^{2n-1}(X;\zp{r})\ar[d,"\rho_\sharp"]\\
   0\ar[r]& \pi^{2n-1}(X;\zz{p})\otimes\zp{}\ar[r,"i_1"]&\pi^{2n-1}(X;\zp{})
 \end{tikzcd}\]
  Since $\pi^{2n-1}(X;\zz{p})$ is  finitely generated, we can write
  \[\pi^{2n-1}(X;\zz{p})=\Z\langle \alpha_1,\cdots,\alpha_k\rangle\oplus\zp{u_1}\langle\beta_1\rangle\oplus\cdots\zp{u_l}\langle \beta_l\rangle.\]
  Then $\{\bar{\alpha}_i,\bar{\beta}_j\}$ form a basis of $\pi^{2n-1}(X;\zz{p})\oplus\zp{r}$. Obviously $(1\otimes\rho)(\gamma_k)\neq 0$ for arbitrary $\gamma_k\in \{\bar{\alpha}_i,\bar{\beta}_j\}$, hence $i_1(1\otimes\rho)(\gamma_k)\neq 0$ for all $k$.
  Assume that $i_r\big(\pi^{2n-1}(X;\zz{p})\oplus\zp{r}\big)\subseteq \pi^{2n-1}(X;\zp{r})$ is not a direct summand, then there exists $\gamma_{k_0}$ satisfying $i_r(\gamma_{k_0})=p \delta$ for some $\delta\in\pi^{2n-1}(X;\zp{r})$. Hence \[p\rho_\sharp(\delta)=\rho_\sharp \circ i_r(\gamma_{k_0})=i_1(1\otimes\rho)(\gamma_{k_0})=0,\]
  which contradicts.
 
 If $\dim(X)\leq 4n-4$, $T_{\phi_r}(X)\cong \pi^{2n-1}(X)\otimes \zp{r}$ follows by Theorem \ref{thmPeterson} (\ref{thmUCT}).
 \end{proof}

\bibliography{Refs}
\bibliographystyle{ijmart}

\end{document}